\documentclass[11pt]{amsart}
\usepackage[T1]{fontenc}

\usepackage{amsmath,amsthm,amsfonts,amssymb,latexsym,mathrsfs,graphicx}
\usepackage[utf8]{inputenc}

\usepackage{hyperref}
\usepackage{longtable}
\usepackage[capitalize, noabbrev]{cleveref}
\usepackage[all]{xy}

\usepackage{tikz-cd}

\usepackage{enumerate}
\usepackage[shortlabels]{enumitem}
\usepackage{color}
\usepackage{comment}
\newcommand{\Z}{\mathbb Z}
\newcommand{\Q}{\mathbb Q}
\newcommand{\C}{\mathbb C}
\newcommand{\R}{\mathbb R}

\newcommand{\Irr}{\operatorname{Irr}}
\newtheorem{theorem}{Theorem}[section]
\newtheorem{corollary}[theorem]{Corollary}
\newtheorem{lemma}[theorem]{Lemma}
\newtheorem{proposition}[theorem]{Proposition}
\theoremstyle{definition}

\newtheorem{remark}[theorem]{Remark}

\newtheorem{maintheorem}{Theorem}
\newtheorem{maincorollary}[maintheorem]{Corollary}

\newenvironment{enumeratei}{\begin{enumerate}[\upshape (a)]}
	{\end{enumerate}}

\def\irr#1{{\rm Irr}(#1)}
\def\cent#1#2{{\bf C}_{#1}(#2)}

\def\syl#1#2{{\rm Syl}_#1(#2)}

\def\oh#1#2{{\bf O}_{#1}(#2)}
\def\zent#1{{\bf Z}(#1)}

\def\gen#1{\langle#1\rangle}
\def\aut#1{{\rm Aut}(#1)}

\def\fit#1{{\bf F}(#1)}
\def\frat#1{{\bf \Phi}(#1)}

\def\irr#1{{\rm Irr}(#1)}

\def\cent#1#2{{\bf C}_{#1}(#2)}

\def\syl#1#2{{\rm Syl}_#1(#2)}

\def\norm#1#2{{\bf N}_{#1}(#2)}
\def\oh#1#2{{\bf O}_{#1}(#2)}

\def\zent#1{{\bf Z}(#1)}

\def\gen#1{\langle#1\rangle}
\def\aut#1{{\rm Aut}(#1)}

\def\fit#1{{\bf F}(#1)}

\def\PSL#1{{\rm PSL}_{2}(#1)}

\def\Z{{\mathbb Z}}
\def\C{{\mathbb C}}
\def\Q{{\mathbb Q}}
\def\irr#1{{\rm Irr}(#1)}

\def\cent#1#2{{\bf C}_{#1}(#2)}
\def\syl#1#2{{\rm Syl}_#1(#2)}
\def\oh#1#2{{\bf O}_{#1}(#2)}

\def\zent#1{{\bf Z}(#1)}

\def\ker#1{{\rm ker}(#1)}
\def\norm#1#2{{\bf N}_{#1}(#2)}

\def \mod#1{\, {\rm mod} \, #1 \, }

\mathchardef\coso="2023

\usepackage[normalem]{ulem}
\usepackage{cancel}
\usepackage[capitalize, noabbrev]{cleveref}

\newcommand{\GEN}[1]{\left\langle #1 \right\rangle}

\newcommand{\ngr}[1]{\Gamma_{\rm{N}}(#1)}
\newcommand{\ngk}[1]{\Gamma_{\rm{GK}}(#1)}

\begin{document}
	
	\title{The N-prime graph and the Subgroup Isomorphism Problem}
	
\begin{abstract} We introduce a directed graph related to a group $G$, which we call the \emph{N-prime graph} $\ngr G$ of $G$ and which is a refinement of the classical Gruenberg-Kegel graph. The vertices of $\ngr G$ are the primes $p$ such that $G$ has an element of order $p$, and, for distinct vertices $p$ and $q$, the arc $q\rightarrow p$ is in the graph if and only if $G$ has a subgroup of order $p$ whose normalizer in $G$ has an element of order $q$. Generalizing some known results about the Gruenberg-Kegel graph, we prove that the group $V(\Z G)$ of the units with augmentation $1$ in the integral group ring $\Z G$ has the same N-prime graph as $G$ if $G$ is a finite solvable group, and we reduce to almost simple groups the problem of whether $\ngr{V(\Z G)}=\ngr G$ holds for any finite group $G$. We also prove that $\ngr{V(\Z G)}=\ngr G$ if $G$ is almost simple with socle either an alternating group, or $\PSL{r^f}$ with $r$ prime and $f\le 2$. Finally, for $G$ solvable we obtain some stronger results which give a contribution to the Subgroup Isomorphism Problem. More precisely, we prove that if $V(\Z G)$ contains a Frobenius subgroup $T$ with kernel of prime order and complement of prime-power order, then $G$ contains a subgroup isomorphic to $T$.
\end{abstract}

\author[E. Pacifici]{Emanuele Pacifici}
\address{Emanuele Pacifici, Dipartimento di Matematica e Informatica U. Dini,\newline
Universit\`a degli Studi di Firenze, viale Morgagni 67/a,
50134 Firenze, Italy.}
\email{emanuele.pacifici@unifi.it}

\author[\'A. del R\'io]{\'Angel del R\'io}
\address{\'Angel del R\'io, Departamento de Matemáticas,
Universidad de Murcia, Campus de Espinardo, 30100 Murcia, Spain.}
\email{adelrio@um.es}

\author[M. Vergani]{Marco Vergani}
\address{Marco Vergani, Dipartimento di Matematica e Informatica U. Dini,\newline
Universit\`a degli Studi di Firenze, viale Morgagni 67/a,
50134 Firenze, Italy.}
\email{marco.vergani@unifi.it}

\thanks{The first and third authors are partially supported by INdAM-GNSAGA, and by the European Union-Next Generation EU, Missione 4 Componente 1, CUP B53D23009410006, PRIN 2022 2022PSTWLB - Group Theory and Applications. The second author is partially supported by Grant PID2024-155576NB-I00 funded by MICIU/AEI/10.13039/501100011033/FEDER, UE, and by Grant 22004/PI/22 funded by Fundación Séneca of Región de Murcia, Spain.}    


\maketitle

\section{Introduction}

This paper deals with the following general question concerning finite groups: how much information about a finite group $G$ can be obtained from the ring theoretical properties of its integral group ring $\Z G$? The history of this question began when G. Higman proved in his Ph.D. dissertation that if $G$ is abelian, then every group $H$ such that $\Z G$ and $\Z H$ are isomorphic must be isomorphic to $G$ \cite{Higman1940Thesis,Higman1940Paper}. 
A. Whitcomb extended the latter result to the case when $G$ is metabelian \cite{WhitcombThesis}, and later A. Weiss proved the same property for nilpotent groups \cite{Weiss1991}.
For a while it was believed that this could be true for every finite group but, in 2001,  M. Hertweck discovered two non-isomorphic groups with isomorphic integral group rings \cite{Hertweck2001}. 
Nevertheless, there is still much evidence that significant information about $G$ is encoded in $\Z G$,  and mostly in the group $V(\Z G)$ of units of $\Z G$ of augmentation 1.
A detailed discussion about the interplay between properties of $G$ and $V(\Z G)$ can be found in the survey \cite{MdR2019} and in \cite{JR}.

The \emph{First Zassenhaus Conjecture} (formulated by H.J. Zassenhaus in 1974) is an iconic problem in this research field, which attracted the interest of many authors in the past years; it predicts that every torsion unit of $\Z G$ is conjugate in $\Q G$ to an element of the form $\pm g$ for a suitable $g\in G$. 
This was proved in many special cases (see e.g. \cite{BachleHermanKonovalovMargolisSingh2018, CaicedoMargolisdelRio2013,Hertweck2006,Weiss1991}) but, in 2018, F. Eisele and L. Margolis provided a counterexample \cite{EiseleMargolis18}. 
As a related concept, we recall that the \emph{Gruenberg-Kegel graph} of a group $G$ (also known as the \emph{prime graph} of $G$) is the simple undirected graph $\ngk{G}$ whose vertex set is the set $\pi(G)$ of the prime numbers dividing the order of some torsion element of $G$, and, for distinct vertices $p$ and $q$, the edge $q-p$ is in the graph if and only if $G$ has an element of order $pq$. 
Now, if the First Zassenhaus Conjecture holds for a given group $G$, then clearly $V(\Z G)$ and $G$ have the same Gruenberg-Kegel graph. In view of this, W. Kimmerle proposed in \cite{Kim06} the so-called \emph{Prime Graph Question}: is it true that, for every finite group $G$, the graphs $\ngk{V(\Z G)}$ and $\ngk{G}$ are equal? The Prime Graph Question is usually abbreviated as {\rm (PQ)}, referring to the common use of $p$ and $q$ for two different primes. It can be interpreted as a weaker form of the First Zassenhaus Conjecture, and there is no known counterexample for it.
Kimmerle proved in \cite{Kim06} that the answer is affirmative if $G$ is a solvable group, and in fact the problem has been reduced to almost simple groups by Kimmerle and A. Konovalov in \cite{KK}; since then, several almost simple groups have been investigated from this point of view (see \cite[Theorem~2.6]{MdR2019} for a wide list of results on the Prime Graph Question for almost simple groups; see also \cite{EiseleMargolis25}).

\smallskip
In this paper we introduce a different graph related to a group $G$, which we call the \emph{N-prime graph} $\ngr G$ of $G$. We define it as the simple directed graph whose vertex set  is $\pi(G)$ and, for distinct vertices $p$ and $q$, the arc $q\rightarrow p$ is in the graph if and only if there exists an element $x$ in $G$ of order $p$ such that $\GEN x$ is normalized by an element $y$ in $G$ of order $q$. As observed in Remark~\ref{GK}, the Gruenberg-Kegel graph of $G$ can be obtained from $\ngr G$ replacing every double arrow $q\leftrightarrows p$ by the edge $q-p$ and deleting all single arrows; on the other hand, the converse is not true because it is easy to find examples of groups having the same Gruenberg-Kegel graph but different N-prime graphs. So, this graph in principle encodes more detailed information about the structure of the group $G$, and we can consider the following strengthening of the Prime Graph Question.

We say that \emph{(NPQ) holds} for a finite group $G$ if $\ngr G=\ngr{V(\Z G)}.$

\medskip
\begin{quote}
{\bf{N-Prime Graph Question:}}    
Is it true that (NPQ) holds for every finite group? 
\end{quote}
\medskip 

The first result of this paper is a generalization to (NPQ) of the aforementioned reduction by Kimmerle and Konovalov for {\rm (PQ)}.

\begin{maintheorem}\label{NPQreduction}
Let $G$ be a finite group, and assume that (NPQ) holds for every almost simple homomorphic image of $G$. Then {\rm (NPQ)} holds for $G$. 
\end{maintheorem}

The corresponding generalization of the main result in \cite{Kim06} is then derived at once.

\begin{maincorollary}\label{NPQsolvable}
Let $G$ be a finite solvable group. Then (NPQ) holds for $G$.
\end{maincorollary}

In view of Theorem~\ref{NPQreduction}, we explore the validity of (NPQ) for some families of almost simple groups. In particular, we prove that (NPQ) holds for rational groups, and for almost simple groups whose socle is an alternating group or a group of the form ${\rm{PSL}}_2(r^f)$ where $r$ is a prime and $f\in\{1,2\}$ (see Theorem~\ref{rational}, Theorem~\ref{alternating} and Theorem~\ref{PSL}, respectively). As a consequence, we obtain the following.

\begin{maintheorem} \label{NPQAnPSL} Let $G$ be a finite group. Assume that every almost simple homomorphic image of $G$ is a rational group, or its socle is an alternating group, or its socle is isomorphic to ${\rm{PSL}}_2(r^f)$ for a prime $r$ and $f\in\{1,2\}$. Then {\rm (NPQ)} holds for $G$.
\end{maintheorem}

Another interesting problem in the present context, somewhat ``intermediate" between the First Zassenhaus Conjecture and the Prime Graph Question, is the \emph{Spectrum Problem}: is it true that, whenever $V(\Z G)$ has an element of a given order $n$, then $G$ has an element of order $n$ as well? It has been proved by Hertweck in \cite{HertweckOrdersSolvable} that the answer is affirmative if the group $G$ is solvable, and there is no known counterexample.  By Corollary~4.1 in \cite{CohnLivingstone1965}, we know that if $V(\Z G)$ has a cyclic subgroup $T$ of prime-power order, then also $G$ has a subgroup isomorphic to $T$. Note that the Spectrum Problem (respectively, the Prime Graph Question) can be restated as follows: is it true that if $V(\Z G)$ has a cyclic subgroup $C_n$ of order $n$ (respectively, of order $pq$), then so does $G$? So these questions ask for the isomorphism types of cyclic subgroups of $G$ and $V(\Z G)$. On the other hand, (NPQ) deals with subgroups of type $C_p\rtimes C_q$. We can frame all these questions in the following general context.

\medskip
\begin{quote} {\bf{Subgroup Isomorphism Problem (SIP):}}
Let $T$ be a finite group. Is it true that if $G$ is a finite group such that $V(\Z G)$ has a subgroup isomorphic to $T$, then also $G$ has a subgroup isomorphic to $T$?
\end{quote}

\medskip

As non-isomorphic groups may have isomorphic integral group rings, of course the (SIP) has in general a negative solution. However, the previous discussion shows that under some conditions the answer is positive. Besides the results above on the case where $T$ is cyclic, Kimmerle observed that the (SIP) has a positive answer for $T=C_2\times C_2$ and Hertweck extended that to $T=C_p\times C_p$ where $p$ is a prime \cite{HertweckCpCp}. Other results for the (SIP) can be found in~\cite{DokuchaevJuriaans1996,Kim15,M2017}. 

Care should be taken in considering {\rm (NPQ)} as part of the (SIP). Let $T$ be the Frobenius group $C_p\rtimes C_q$. Suppose that (NPQ) holds for $G$ and $V(\Z G)$ has a subgroup isomorphic to $T$. If, moreover, $G$ does not have elements of order $pq$, then $G$ has a subgroup isomorphic to $T$. However, if $G$ has elements of order $pq$, then the existence of a subgroup of $G$ isomorphic to $T$ is not granted by (NPQ).

Our next result shows that such a subgroup actually exists, under the assumption that $G$ is solvable. More generally, we prove the following theorem which is a contribution to the (SIP) for $T$ a Frobenius group of type $C_p\rtimes C_{q^k}$. 

\begin{maintheorem}\label{SIPsolvable}
	Let $G$ be a finite solvable group and let $T$ be a Frobenius group of the form $C_p\rtimes C_{q^k}$ with $p$ and $q$ primes. If $V(\Z G)$ has a subgroup isomorphic to $T$, then so does~$G$.
\end{maintheorem}

As already mentioned, by \cite{HertweckOrdersSolvable}, if $G$ is solvable and $V(\Z G)$ has a subgroup isomorphic to $C_p\times C_{q^k}$, then so does $G$. This, together with the previous result, solves the (SIP) for subgroups of the form $C_p\rtimes C_{q^k}$ in the two extreme situations when the action of $C_{q^k}$ on $C_p$ is trivial or faithful. So, we can consider the following.

\medskip
\begin{quote}{\bf{Question:}} Let $G$ be a finite solvable group, and assume that $V(\Z G)$ has a subgroup $T$ of the form $C_{p^h}\rtimes C_{q^k}$, where $p$ and $q$ are different primes. Is it true that $G$ has a subgroup isomorphic to $T$?
\end{quote} 
\medskip

The last result of this paper yields an affirmative answer under the assumption that the derived subgroup of $G$ is cyclic.

\begin{maintheorem}\label{SIPmetacyclic}
	Let $G$ be a finite group such that $G'$ is cyclic, and let $T$ be a group of the form $C_{p^h}\rtimes C_{q^k}$ with $p$ and $q$ different primes. If $V(\Z G)$ has a subgroup isomorphic to $T$, then so does $G$.
\end{maintheorem}

\section{Preliminary notation}

In this brief preliminary section we introduce some notation and concepts that will come into play. 

For a finite group $G$ and a commutative ring $R$, we denote by $V(R G)$ the group of units having augmentation $1$ in the group ring $R G$. 
Note that, if $M$ is a normal subgroup of $G$, then the natural projection of $G$ onto $G/M$ can be extended by linearity to a ring homomorphism $\varphi_M:R G\rightarrow R(G/M)$, and the restriction of $\varphi_M$ to $V(RG)$ yields a group homomorphism (that we will still denote by $\varphi_M$) to $V(R(G/M))$. We will refer to this $\varphi_M$ as to the ``natural homomorphism'' associated with $M$ (both as a ring homomorphism on $RG$ and as a group homomorphism on $V(RG)$). The subscript $M$ will be usually dropped when the context yields no ambiguity.

We also recall that, given an element $v=\sum_{g\in G}r_g g\;$ in $RG$ (where the coefficients $r_g$ are in $R$), the \emph{partial augmentation} of $v$ with respect to the element $b\in G$ is $$\varepsilon_b(v)=\sum_{g\sim b}r_g\;,$$ where we use the notation $\sim$ for conjugation in $G$ (clearly, $\varepsilon_b(v)$ depends on the conjugacy class of $b$ in $G$ rather than on $b$ itself). Observe that $\varepsilon_b:RG\rightarrow R$ is an $R$-linear map. Moreover, as easily checked, we have $$[RG,RG]=\{v\in RG\mid\varepsilon_g(v)=0\;\;{\text{ for every }}g\in G\}$$ (see \cite[(41.1), page 237]{Seh93}), which implies that $\varepsilon_b(u)=\varepsilon_b(v)$ if $u$ and $v$ are conjugate in $RG$. 

Note that, if $p$ and $q$ are distinct primes, then a group of the form $C_p\rtimes C_q$ is either cyclic or a Frobenius group.

Finally, recall that a finite group $G$ is called a \emph{$2$-Frobenius group} if it has normal subgroups $F$ and $K$ such that $K$ is a Frobenius group with kernel $F$, and $G/F$ is a Frobenius group with kernel $K/F$.

\section{The N-Prime Graph Question: Theorem~\ref{NPQreduction}}

We begin our study of the N-prime graph, that was defined in the introduction, and we first clarify its relationship with the Gruenberg-Kegel graph in the following easy observation.

\begin{remark} \label{GK}
 We observe that the Gruenberg-Kegel graph of a group $G$ has the same vertex set as the N-prime graph of $G$, and its edges can be deduced from those of the N-prime graph as follows: $\ngk G$ has an edge joining two vertices $p$ and $q$ if and only if both $p\rightarrow q$ and $q\rightarrow p$ are arcs of $\ngr G$. Indeed, it is clear from the definitions that if $p-q$ is an edge of $\ngk{G}$, then both $q\rightarrow p$ and $p\rightarrow q$ are arcs of $\ngr G$. On the other hand, if $p$ and $q$ are not adjacent in $\ngk G$ but $\ngr G$ contains the arc $q\rightarrow p$, then $q$ divides $p-1$, thus clearly $\ngr G$ cannot contain the arc $p\rightarrow q$.
 
 An alternative way to describe $\ngk{G}$ from $\ngr G$ is as the undirected graph with the same vertices as $\ngr G$ and edges $p-q$ if and only if $p>q$ and $p\to q$ is an arc of $\ngr G$.

However, in general $\ngr G$ cannot be derived from $\ngk G$: for example, the groups ${\rm{PSL}}_2(7)$ and ${\rm{PSL}}_2(8)$ have the same Gruenberg-Kegel graph $(2\quad 3\quad 7)$, but $\ngr{{\rm{PSL}}_2(7)}$ is $(2\rightarrow 3\rightarrow7)$ whereas $\ngr{{\rm{PSL}}_2(8)}$ is $(7\leftarrow 2\rightarrow 3)$. For an example involving solvable groups, we can consider the 2-Frobenius group $(C_3^6\rtimes C_7)\rtimes C_3$ and the Frobenius group $C_3^6\rtimes C_7$: for both of them the Gruenberg-Kegel graph is disconnected, and so is the N-prime graph of the latter, but the N-prime graph of the former is $(3\rightarrow 7)$.
\end{remark}

It is straightforward that, if $H$ is a subgroup of the finite group $G$, then $\ngr H$ is a subgraph of $\ngr G$, so the graph $\ngr G$ is well behaved with respect to subgroups. We will show next that $\ngr G$ behaves well also with respect to factor groups: if $M$ is a normal subgroup of $G$ then $\ngr{G/M}$ is a subgraph of $\ngr G$. To this end, we introduce a lemma that will be useful also for other results in this paper.

\begin{lemma}\label{ActionNT}
Let $p$ and $q$ be primes, and let $k$ be a positive integer such that $q^k$ divides $p-1$.
Let $P$ be a finite $p$-group and let $\beta$ be an automorphism of order $q^k$ of $P$. Then there exists $a\in P$ such that $|a|=p$ and $\beta$ restricts to an automorphism of order $q^k$ of $\GEN{a}$. 
\end{lemma}

\begin{proof} We argue by induction on the order of $P$. If the Frattini subgroup $\frat P$ of $P$ has index $p$ in $P$, then $P$ is a cyclic group and our claim follows; thus we may assume that $|P:\frat P|>p$. By coprimality, the action of $\langle \beta\rangle $ on the factor group $P/\frat P$ is faithful (see Theorem~5.3.5 in \cite{Gor}), so we can view $P/\frat P$ as a completely reducible $\langle\beta\rangle$-module over ${\rm{GF}}(p)$ having at least one faithful irreducible constituent. Denoting this constituent by $P_0/\frat P$ and setting $|P_0/\frat P|=p^n$, we have that $\beta$ can be identified with an element of the multiplicative group ${\rm{GF}}(p^n)^\times$ acting on $P_0/\frat P$ as a multiplication map. But since $|\beta|=q^k$ divides $p-1$, $\beta$ is indeed identified with an element of ${\rm{GF}}(p)^\times$, thus $|P_0/\frat P|$ must be $p$. We conclude that $P_0$ is a proper subgroup of $P$ on which $\beta$ acts as an automorphism of order $q^k$. The desired conclusion follows now by induction.
\end{proof}

\begin{proposition}\label{subgraph}
Let $G$ be a finite group and let $M$ be a normal subgroup of $G$. Then $\ngr{G/M}$ is a subgraph of $\ngr G$.    
\end{proposition}

\begin{proof}
Clearly the vertex set of $\ngr{G/M}$ is contained in that of $\ngr G$. Also, let $p,q\in\pi(G/M)$ be such that $q\rightarrow p$ is an arc of $\ngr{G/M}$; this means that there exists a subgroup $H$ of $G$, containing $M$, with $H/M\cong C_p\rtimes C_q$. If $H/M$ has an element of order $pq$, then the same holds for $H$ and we get the arc $q\rightarrow p$ in $\ngr G$; thus we can assume that $H/M$ is not a direct product, which implies that $q$ is a divisor of $p-1$. Let $K$ be the subgroup of $H$ containing $M$ such that $K/M\in\syl p{H/M}$, and let $P$ be a Sylow $p$-subgroup of $H$. Clearly $P$ is contained in $K$, and by the Frattini argument we get $H=K\norm H P$. Since the $q$-part of $|H|$ is strictly larger than that of $|K|$, we get that $q$ is a divisor of $|\norm H P|$ and we can consider a subgroup $C\subseteq\norm H P$ having order $q$. 

If $C$ centralizes $P$, then we can easily construct an element of order $pq$ of $G$. Otherwise, Lemma~\ref{ActionNT} yields the existence of a subgroup $U$ of $P$ such that $|U|=p$ and $C\subseteq\norm G U$. Now the subgroup $UC$ yields the arc $q\rightarrow p$ in $\ngr G$, and the proof is complete.
\end{proof}

\begin{remark}\label{CpCq2}
It is not true in general that if $M$ is a normal subgroup of $G$ and $G/M$ contains a subgroup $T$ isomorphic to the Frobenius group $C_p\rtimes C_q$, then $G$ contains a subgroup isomorphic to $T$. For example, if $p\equiv 1\, \mod\, q$ then there is a positive integer $r$ of order $q$ modulo $p$, i.e. $r\not\equiv 1\, \mod\, p$ and $r^q\equiv 1\, \mod\, p$. Then the center of the group $G=\GEN{a,b\mid a^p=b^{q^2}=1, a^b=a^r}$ is $M=\GEN{b^q}$, and $G/M\cong T$, but $G$ does not have any subgroup isomorphic to $T$. In this case we have $\ngr G=(p \leftrightarrows q)$ and $\ngk{G}=(p-q)$, while $\ngr{G/M} = (p\leftarrow q)$ and $\ngk{G/M}=(p \quad q)$.
\end{remark}

We are ready to start our discussion about the N-Prime Graph Question, that was presented in the introduction. 
Note that, by \cite[Corollary~4.1]{CohnLivingstone1965}, $\ngr G$ and $\ngr{V(\Z G)}$ have the same set of vertices. Moreover, it is clear that $\ngr G$ is a subgraph of $\ngr{V(\Z G)}$. So, the N-Prime Graph Question amounts to understanding if, for $p,q$ distinct primes in $\pi(G)$, the existence of the arc $q\rightarrow p$ in $\ngr{V(\Z G)}$ implies the existence of the same arc in $\ngr G$. 

\medskip

A finite group $G$ is called \emph{rational} if its ordinary character table contains only rational entries. As is well known, this is equivalent to the fact that every element $x$ of $G$ is conjugate in $G$ to every generator of $\GEN x$, which is also equivalent to the condition 
$\norm G{\GEN x}/\cent G x\cong \aut{\GEN x}$ for every $x\in G$. 
It has been proved in Corollary~H of \cite{BKMdR} that (PQ) holds for rational groups.
Based on this we prove that also (NPQ) holds for such groups.

\begin{theorem}\label{rational}
Let $G$ be a rational group. Then (NPQ) holds for $G$.
\end{theorem}

\begin{proof}
	Let $p$, $q$ be vertices of $\ngr{V(\Z G)}$ and assume that the arc $q\rightarrow p$ is in $\ngr{V(\Z G)}$. If also $p\rightarrow q$ is an arc of $\ngr{V(\Z G)}$, then we know that $p-q$ is an edge of $\ngk{V(\Z G)}$ and the same holds for $\ngk{G}$ by \cite[Corollary H]{BKMdR}. In particular, the arc $q\rightarrow p$ is in $\ngr G$ and we are done. On the other hand, if $p\rightarrow q$ is not an arc of $\ngr{V(\Z G)}$, then $q$ is a divisor of $p-1$; but since $G$ is a rational group, for any element $x\in G$ of order $p$ we have $|\norm G{\GEN x}/\cent G x|=|\aut{\GEN x}|=p-1$, so there exists an element of order $q$ in $\norm G{\GEN{x}}$ and our claim is proved.
\end{proof}

The following result, which collects together Theorem~\ref{NPQreduction} and Corollary~\ref{NPQsolvable}, shows that the N-Prime Graph Question can be reduced to almost simple groups. This extends the corresponding result by Kimmerle and Konovalov for the Prime Graph Question.

\begin{theorem}\label{NPQ}
    Let $G$ be a finite group, and assume that {\rm{(NPQ)}} holds for every almost simple homomorphic image of $G$. Then {\rm (NPQ)} holds for $G$. In particular, {\rm{(NPQ)}} holds for every finite solvable group.
\end{theorem}

\begin{proof}
 As observed in the paragraph after Remark~\ref{CpCq2}, we only have to show that, for $p$ and $q$ distinct primes in $\pi(G)$, the existence of the arc $q\rightarrow p$ in $\ngr{V(\Z G)}$ implies the existence of the same arc in $\ngr G$. 
    For a proof by contradiction, we will assume that $G$ is a minimal counterexample to this fact with respect to the primes $p$ and $q$. Note that, by our assumptions and by Remark~\ref{GK}, if $X$ is any almost simple homomorphic image of $G$ then also the Gruenberg-Kegel graphs $\ngk{X}$ and $\ngk{V(\Z X)}$ coincide; therefore, by \cite{KK}, we have $\ngk{G}=\ngk{V(\Z G)}$, and the presence of both arcs $q\leftrightarrows p$ in $\ngr{V(\Z G)}$ would imply the presence of the edge $q-p$ in $\ngk{G}$. As a consequence both arcs $q\leftrightarrows p$ would be in $\ngr G$, against the fact that $G$ is a counterexample.  So $\ngr{V(\Z G)}$ has the arc $q\rightarrow p$ but not the arc $p\rightarrow q$; in particular, $V(\Z G)$ has a subgroup $T$ which is a Frobenius group of order $pq$, and $q$ is a divisor of $p-1$. 
    
Let $M$ be a minimal normal subgroup of $G$, and observe that the factor group $G/M$ satisfies our assumptions. Denoting by $\varphi:V(\Z G)\to V(\Z (G/M))$ the natural homomorphism, if $p$ does not divide $|M|$, then we get $T\cap \ker\varphi=1$ (see \cite[Lemma~7.5]{Seh93}) and hence $\varphi(T)\cong T$. It follows that $q\rightarrow p$ is an arc of $\ngr{V(\Z(G/M))}$ and, by our minimality assumption, it is also an arc of $\ngr{G/M}$. But then $q\rightarrow p$ is an arc of $\ngr G$ by Proposition~\ref{subgraph}, a contradiction. Therefore, no minimal normal subgroup of $G$ can be a $p'$-group.   

The discussion in the paragraph above implies that the Fitting subgroup $\fit G$ of $G$ is a $p$-group, and we claim that in fact $\fit G$ must be trivial. As in the proof of Proposition~\ref{subgraph}, consider a subgroup $C$ of $G$ having order $q$ and assume $\fit G\neq 1$. If $C$ centralizes $\fit G$ then we can construct an element of order $pq$ of $G$, not our case. Otherwise, we can use Lemma~\ref{ActionNT} and find a subgroup $U$ of $\fit G$ such that $|U|=p$ and $C\subseteq\norm G U$. Now the subgroup $UC$ yields the arc $q\rightarrow p$ in $\ngr G$, again a contradiction. 

To sum up, we have $\fit G=1$ and every minimal normal subgroup of $G$ is nonabelian of order divisible by $p$. Note that the last claim of the statement, concerning solvable groups, is already proved at this stage.

Our next claim is that $G$ cannot have any minimal normal subgroup whose order is divisible also by $q$. In fact, assume the contrary and consider a minimal normal subgroup $M$ of $G$ such that $pq$ divides $|M|$. Since $M\cong S_1\times\cdots\times S_t$ is a direct product of isomorphic nonabelian simple groups, we see that $pq$ divides the order of each $S_i$. Now $t$ must be $1$, as otherwise we could produce an element of order $pq$ in $M$ (hence in $G$) by multiplying an element of order $p$ of $S_1$ with an element of order $q$ of $S_2$. Moreover, if such a subgroup $M$ exists, then it must be the unique minimal normal subgroup of $G$. In fact, an element of order $p$ in any other minimal normal subgroup of $G$ would commute with an element of order $q$ of $M$, again yielding the contradiction that $G$ has an element of order $pq$. So $M$ would be simple, and the unique minimal normal subgroup of $G$. But then $G$ would be an almost simple group, and by our assumptions it would not be a counterexample.  

As another remark, given a minimal normal subgroup $M\cong S_1\times\cdots\times S_t$ of $G$, for any fixed $i\in\{1,\ldots,t\}$ we claim that $q$ does not divide $|\norm G{S_i}|$. Assuming the contrary, we could take an element $y$ of order $q$ in $\norm G{S_i}$. Recalling that $S_i$ is a $q'$-group having an order divisible by $p$, we could find a (nontrivial) Sylow $p$-subgroup $P$ of $S_i$ such that $y$ normalizes $P$ (see \cite[8.2.3(a)]{KS}), and the same argument as in the third paragraph of this proof leads to a contradiction.  

Now, let $M\cong S_1\times\cdots\times S_t$ and define $$K=\bigcap_{i=1}^t\norm G{S_i};$$ in view of the paragraph above, $K$ is a $q'$-group.  Since the $S_i$ are non-abelian simple groups, conjugation yields a permutation action of $G$ on the set  $\{S_1,\ldots,S_t\}$ with kernel $K$. So, $G/K$ embeds in the symmetric group ${\rm{Sym}} (\{S_1,\ldots,S_t\})\cong{\rm{Sym}}(t)$ and, if $y\in G$ is an element of order $q$, then $yK$ can be identified (up to renumbering) with a permutation containing the cycle $(1,2,\ldots,q)$. Take an element $x_1\in S_1$ of order $p$ and, setting $x_i=x_1^{y^{i-1}}$ for all $i\in\{1,\ldots, q\}$, observe that $x_i$ lies in $S_{i}$; in particular, the elements $x_i$ pairwise centralize each other. Defining $x=x_1\cdots x_q$, it is not difficult to see that $x$ and $y$ commute, hence again we get that $xy$ is an element of order $pq$ in $G$. This is the final contradiction that completes the proof.
\end{proof}

The following corollary can be immediately derived from the theorem above and from Theorem~\ref{rational}. 

\begin{corollary}\label{rationalsocle} Let $G$ be a finite group. If every almost simple homomorphic image of $G$ is a rational group, then {\rm (NPQ)} holds for $G$.
\end{corollary}

\section{The N-Prime Graph Question: Theorem~\ref{NPQAnPSL}}

In this section we prove that (NPQ) holds for every almost simple group whose socle is isomorphic to an alternating group, or to ${\rm{PSL}}_2(r^f)$ where $r$ is a prime and $f\in\{1,2\}$ (Theorem~\ref{alternating} and Theorem~\ref{PSL}, respectively); this, together with Theorem~\ref{NPQ} and Corollary~\ref{rationalsocle}, will conclude the proof of  Theorem~\ref{NPQAnPSL} of the introduction.

Before stating the next result we mention that, by Theorem~1.1 of \cite{BM} and Theorem~A of \cite{BM2017}, for every almost simple group $G$ with a socle isomorphic to an alternating group or to ${\rm{PSL}}_2(r^f)$ where $r$ is a prime and $f\in\{1,2\}$, we have $\ngk{V(\Z G)}=\ngk{G}$.

\begin{theorem}\label{alternating}
 Let $G$ be an almost simple group whose socle is isomorphic to the alternating group ${\rm A}_n$, for $n\geq 5$. Then {\rm{(NPQ)}} holds for $G$. 
\end{theorem}

\begin{proof}
Let $p$ and $q$ be primes in $\pi(G)$ and, by means of contradiction, suppose that $q\to p$ is an arc of $\ngr{V(\Z G)}$ but not of  $\ngr{G}$. Then, taking into account Theorem~\ref{rational}, $G$ is not isomorphic to the symmetric group ${\rm S}_n$, so we have $G\cong {\rm A}_n$ unless $n=6$ (and $|G|$ divides $4|{\rm A}_n|$). Moreover, $G$ does not have elements of order $pq$ and hence, by the main theorem of \cite{BM}, neither does $V(\Z G)$. Thus $V(\Z G)$ has a subgroup $T=\GEN{u}\rtimes \GEN{v}$ with $|u|=p$, $|v|=q$ and $u^v=u^r$ for a suitable integer $r$ such that $r\not\equiv 1\, \mod \,p$ and $r^q\equiv 1\, \mod\, p$.

By Corollary~4.1 of \cite{CohnLivingstone1965}, ${\rm A}_n$ has then an element $a$ of order $p$; in particular we have $n\ge p$, and we may take $a=(1,2,\dots,p)$. If $n\geq p+2$, then all the elements of ${\rm A}_n$ with the same cyclic type as $a$ lie in a single conjugacy class of ${\rm A}_n$, hence $a^r$ is conjugate to $a$ in ${\rm A}_n$. However, if $q$ is odd, then the same conclusion holds also for $n\in\{p,p+1\}$, because the map $x\mapsto x^r$ on ${\rm A}_n$ induces a permutation of odd order on the set $\mathcal{C}$ of conjugacy classes of elements having order $p$ in ${\rm A}_n$ and, in this case, we have $|\mathcal{C}|=2$. Now, if $a^r$ is conjugate to $a$ in ${\rm A}_n$, then we can find an element $b\in \norm{{\rm A}_n}{\gen a}$ with $|b|=q$, against our assumption. We conclude that $n$ lies in $\{p,p+1\}$ and $q=2$. Moreover, $T$ is a dihedral group of order $2p$ and $a$ is not conjugate to $a^{-1}$ in ${\rm A}_n$, i.e., $a$ is not a real element of ${\rm A}_n$. Note that the latter property does not hold if $p\equiv 1\,\mod\, 4$, hence $n$ cannot be $6$ and we have $G\cong {\rm A}_n$. 

Since the element $a\in G$ is not real, we can take an irreducible character $\chi$ of $G$ such that $\chi(a)\not\in \R$, and we denote by $x$ the (non-zero) imaginary part of $\chi(a)$. On the other hand, as $u$ is conjugate to $u^{-1}$ in $T$, the linear expansion $\widehat{\chi}$ of $\chi$ to $\Z G$ takes a real value on $u$.

%Fix a complex primitive $p$-th root of the unity $\zeta$ and consider the character $\psi$ of $T$ induced by the linear character of $\GEN{u}$ mapping $u$ to $\zeta$. Then  $\psi(u^i)=\zeta^i+\zeta^{-i}$ and $\psi(u^iv)=0$ for every $i\in \{0,1,\dots,p-1\}$. 
%As $\chi_T$ and $\psi$ are characters of $T$, the inner product $\GEN{\psi,\chi_T}$ is a non-negative integer, so its imaginary part is $0$. 

Now, $G$ has exactly two conjugacy classes of elements of order $p$, represented by $a$ and $a^{-1}$. 
By the Berman-Higman Theorem and \cite[Theorem~2.3]{Hertweck2007}, if $g\in G$ is such that $\varepsilon_g(u)\ne 0$, then $g$ is conjugate to $a$ or $a^{-1}$.
Therefore $\varepsilon_a(u)+\varepsilon_{a^{-1}}(u)=1$ and 
		\[{\widehat{\chi}}(u)=\varepsilon_a(u)\chi(a)+\varepsilon_{a^{-1}}(u)\overline{\chi(a)} = 
		\varepsilon_a(u)\chi(a)+(1-\varepsilon_a(u))\overline{\chi(a)}=\varepsilon_a(u)(\chi(a)-{\overline{\chi(a)}})+{\overline{\chi(a)}}.\]
Comparing the imaginary parts of the two sides, we get $0=x\cdot(2\varepsilon_a(u)-1)$; but $x\neq 0$, and $\varepsilon_a(u)$ is then an integer such that $2\varepsilon_a(u)-1=0$. This is the final contradiction which completes the proof.
\end{proof}

We move next to almost simple groups whose socle is a projective special linear group. The following general lemma will be crucial in our discussion, and it might be useful to prove that (NPQ) holds in other cases.

\begin{lemma}\label{independence}
Let $G$ be a finite group, let $p$ and $q$ be distinct primes, and let $\zeta$ be a complex primitive $p$-th root of unity. Also, let $\sigma$ be an automorphism of $\Q_p=\Q(\zeta)$ of order $q$, let $F=\{x\in \Q_p : \sigma(x)=x\}$ and let $r$ be an integer such that $\sigma(\zeta)=\zeta^r$.
Denoting by $A$ a set of representatives for the conjugacy classes of elements of order $p$ in $G$, suppose that $G$ has a class function $\chi:G\to \C$ satisfying the following conditions:
\begin{enumeratei}
	\item $\chi$ is an $F$-linear combination of irreducible characters of $G$.
	\item For every $a\in A$, we have that $[\Q_p:\Q(\chi(a))]$ is not divisible by $q$.	
	\item The complex numbers $\chi(a)$, for $a\in A$, are linearly independent over~$\Q$.
\end{enumeratei}
Then $V(\Z G)$ does not have a Frobenius subgroup of the form $C_p\rtimes C_q$. 
\end{lemma}

\begin{proof}
Suppose, by means of contradiction, that $V(\Z G)$ has a Frobenius subgroup $T$ of the form $C_p\rtimes C_q$. So, $T=\GEN{u}\rtimes \GEN{v}$ with $|u|=p$, $|v|=q$, and we can assume $u^v=u^r$. 

Observe that the $r$-th power map defines a permutation $\varrho$ on the set $\{a^G\mid a\in A\}$ (consisting of the conjugacy classes of elements having order $p$ in $G$), given by $\varrho(a^G)= (a^r)^G$; as $r^q\equiv 1\,\mod\,p$, this permutation has an order dividing $q$. Moreover, for $\varphi\in\irr G$ and $a\in A$, we get $\varphi(a)\in\Q_p$ and $\sigma(\varphi(a))=\varphi(a^r)$.
By our assumption in (a), we have $\chi=\sum_{\varphi\in \Irr(G)} c_\varphi \varphi$ for suitable $c_\varphi\in F$, and hence $\sigma(\chi(a))=\chi(a^r)$ for every $a\in A$. Furthermore, taking into account that $[\Q_p:F]=q$, for every $a\in A$
the assumption in (b) yields $\chi(a)\not\in F$; in other words, we have $\sigma(\chi(a))\ne \chi(a)$ and, in particular, $a^r$ is not conjugate to  $a$ in $G$. As a consequence, all the orbits of the permutation $\varrho$ have cardinality $q$.

Every character $\varphi\in\irr G$ extends uniquely to a linear map $\C G\to \C$, which we also denote by $\varphi$.
As $u$ and $u^r$ are conjugate in $V(\Z G)$, and hence in $\C G$, it follows that $\varphi(u)=\varphi(u^r)$. Also, it is not difficult to see that we have $\varphi(u^r)=\sigma(\varphi(u))$. 
Similarly, $\chi$ extends uniquely to a linear map $\C G\to \C$, which we also denote $\chi$.
Using the expression of $\chi$ as a linear combination of irreducible characters of $G$ we obtain  $\sigma(\chi(u))=\chi(u^r)=\chi(u)$. 
On the other hand, by the Berman-Higman Theorem and \cite[Theorem~2.3]{Hertweck2007}, $\varepsilon_g(u)=0$ for every $g\in G$ of order different from $p$.
Hence, since $\chi$ is a class function, we have 
	\[\sum_{a\in A} \varepsilon_{a^{r}}(u) \chi(a) =\sum_{a\in A} \varepsilon_a(u) \sigma^{-1}(\chi(a)) = \sigma^{-1}(\chi(u))=\chi(u)=\sum_{a\in A} \varepsilon_a(u) \chi(a).\]
Since the partial augmentations are integers, our assumption in (c) implies that $\varepsilon_{a_1}(u)=\varepsilon_{a_2}(u)$ whenever $a_1^G$ and $a_2^G$ belong to the same $\varrho$-orbit. Thus, denoting by $B$ a subset of $A$ such that $\{b^G\mid b\in B\}$ is a set of representatives of these orbits, we get
	$$1=\sum_{a\in A} \varepsilon_a(u)=q\sum_{b\in B} \varepsilon_b(u),$$
a contradiction. 
\end{proof}

\begin{theorem}\label{PSL}
 Let $G$ be an almost simple group whose socle $S$ is isomorphic to ${\rm{PSL}}_2(r^f)$, for a prime $r$ and $f\in\{1,2\}$. Then {\rm{(NPQ)}} holds for $G$. 
\end{theorem}

\begin{proof}
Since ${\rm{PSL}}_2(2^2)$ is isomorphic to ${\rm A}_5$, which has been already treated in Theorem~\ref{alternating}, we can assume that $r$ is an odd prime. We start by recalling that the order of $S$ for $r\neq 2$ is ${r^f(r^f-1)(r^f+1)}/2$, and the order of ${\rm{Out}}(S)$ is $2f$; in fact, the automorphism group of $S$ is isomorphic to $S\GEN{\delta, \phi}$, where $\delta$ is a diagonal automorphism such that $\delta^2$ is an inner automorphism, and $\phi$ is a field automorphism of order $f$ (see for instance \cite{White}). Therefore $\pi(G)=\pi(S)=\pi^{-}\cup\pi^{+}\cup\{r\}\cup\{2\}$, where $\pi^{-}$ and  $\pi^{+}$ are the sets of odd prime divisors of $r^f-1$ and of $r^f+1$, respectively. 

By \cite[II.8.27]{Hu}, the group $S$ has dihedral subgroups both of order $r^f-1$ and of order $r^f+1$; also, it has subgroups that are Frobenius groups with an elementary abelian kernel of order $r^f$ and cyclic complements of order $(r^f-1)/2$. In view of this fact, if $r^f\equiv 1\,\mod\, 4$, then $\ngr{S}$ contains the following arcs: $s\leftrightarrows t$ for any choice of $s$ and $t$ in $\{2\}\cup \pi^{-}$ or in $\pi^{+}$, and $2\rightarrow s$ for any $s\in\{r\}\cup\pi^{+}$. Similarly, if $r^f\equiv -1\,\mod\, 4$, then $\ngr{S}$ contains the arcs $s\leftrightarrows t$ for any choice of $s$ and $t$ in $\{2\}\cup \pi^{+}$ or in $\pi^{-}$, and $2\rightarrow s$ for any $s\in\pi^{-}$. Moreover, in both cases $\ngr{S}$ has the arcs $s\rightarrow r$ where $s\in\pi^{-}$ is a divisor of $r-1$. 

Now, all these arcs are clearly in $\ngr G$, and also $2\rightarrow r$ is in $\ngr G$ unless $r^f\equiv -1\,\mod\,4$ and $G=S$; in fact, it is not difficult to check that $G$ has always a (possibly abelian) subgroup of the form $C_r\rtimes C_2$ whenever $G$ is not simple. Furthermore, every double arc in $\ngr{V(\Z G)}$ is also in $\ngr G$ by \cite[Theorem~5.5]{MdR2019} and Remark~\ref{GK}. Since, by obvious arithmetical reasons, $\ngr{V(\Z G)}$ cannot contain arcs of the kind $s\rightarrow r$ for any $s\in\pi^+$, or $s$ in $\pi^-$ not dividing $r-1$, the desired conclusion will be obtained by proving the following properties.
\begin{enumeratei}
\item $\ngr{V(\Z G)}$ does not contain any arc of the form $q\rightarrow p$, for any choice of a pair $(p,q)$ in $\pi^+\times\pi^-$ or in $\pi^-\times\pi^+$.
\item $\ngr{V(\Z G)}$ does not contain the arc $2\rightarrow r$ if $G= S$ and $r^f\equiv -1\,\mod\,4$.
\end{enumeratei}

We work next to show that (a) holds, and we start from the case $G=S$. Assuming for the moment that we have $r^f\equiv 1\,\mod\,4$, we consider first the situation when $p$ lies in $\pi^+$ and $q$ in $\pi^-$. In view of the character table of $S$, as displayed in Table~2 on page~10 of \cite{Hertweck2007}, a set of representatives for the conjugacy classes of elements of order $p$ in $G$ is $$\left\{b^m\mid m\in\Z,\;{\text{gcd}}\left(\frac{r^f+1}{2},m\right)=\frac{r^f+1}{2p}\,{\text{ and }}\, 1\leq m\leq \frac{r^f-1}{4}\right\},$$ where $b\in G$ is a suitable element of order $(r^f+1)/2$. Observe that the two conditions on the integer $m$ yield $m=k\cdot (r^f+1)/2p\,$ for some $k\in\Z$ with $$k\leq\frac{r^f-1}{4}\cdot\frac{2p}{r^f+1}\leq \frac{p}{2},$$ which indeed implies $k\leq(p-1)/2$. Moreover, setting $\sigma$ to be a complex primitive $\frac{r^f+1}{2}$-th root of unity and adopting the notation of the aforementioned table, the value of the irreducible character $\theta_1$ (of degree $r^f-1$) on $b^m$ is $-(\sigma^m+\sigma^{-m})=-(\zeta^k+\zeta^{-k})$, where $\zeta=\sigma^{(r^f+1)/2p}$ is a primitive $p$-th root of unity. 

Since $\{\zeta^k+\zeta^{-k}\mid 1\leq k\leq (p-1)/2\}$ is a basis for $\Q(\zeta+\zeta^{-1})$ over $\Q$, the values of $\theta_1$ on the given set of representatives for the conjugacy classes of elements of order $p$ are linearly independent over $\Q$, and clearly $q$ does not divide $2=|\Q_p:\Q(\theta_1(b^m))|$ for any such representative $b^m$. We are then in a position to apply Lemma~\ref{independence}, obtaining the desired conclusion. 

Still assuming $G=S$ and $r^f\equiv 1\,\mod\,4$, let us consider the case when $p$ lies in $\pi^-$ and $q$ in $\pi^+$. We can apply the same method as above, this time using the irreducible character $\chi_1\in\irr G$ (of degree $r^f+1$) in place of $\theta_1$. The conjugacy classes of elements of order $p$ are now represented by the elements $a^\ell$, where $a\in G$ is a suitable element of order $(r^f-1)/2$, and $\ell$ ranges over all integers that satisfy the following conditions: $${\text{gcd}}\left(\frac{r^f-1}{2},\ell\right)=\frac{r^f-1}{2p},\quad{\text{and}}\quad 1\leq \ell\leq \frac{r^f-1}{4}.$$ As above, this yields $\ell=k\cdot (r^f-1)/2\,$ with $1\leq k\leq (p-1)/2$; denoting by $\rho$ a primitive $\frac{r^f-1}{2}$-th root of unity and setting $\zeta=\rho^{(r^f-1)/2}$, we see that the value of $\chi_1$ on $a^\ell$ is $\zeta^k+\zeta^{-k}$, and the desired conclusion follows by Lemma~\ref{independence}, as above.

A slight variation of our argument so far yields that (a) is true for $G=S$ also when $r^f\equiv -1\,\mod\, 4$, both in the case $p\in\pi^+$, $q\in\pi^-$ and vice-versa. Thus, the claim in (a) is proved when $G=S$.

Let now $G$ be any almost simple group with a socle $S$ isomorphic to ${\rm{PSL}}_2(r^f)$. If the characters $\theta_1$ and $\chi_1$ can be extended to characters $\widehat{\theta_1}$ and $\widehat{\chi_1}$ in $\irr G$, then it is easily seen that we can argue exactly as in the case $G=S$, using $\widehat{\theta_1}$ and $\widehat{\chi_1}$ in place of $\theta_1$ and $\chi_1$. On the other hand, if either $\theta_1$ or $\chi_1$ does not extend to $G$, then Theorem~6.6 and Theorem~6.7 of \cite{White} yield that $G$ is isomorphic to one of the following groups: 
\begin{itemize}
\item $\aut{{\rm{PSL}}_2(r^2)}={\rm{PSL}}_2(r^2)\GEN{\delta,\phi}$, 
\item ${\rm{PSL}}_2(r^2)\GEN\phi$, 
\item ${\rm{PSL}}_2(r^2)\GEN{\delta\phi}$.
\end{itemize}

So, let us define $\xi$ to be either $\theta_1$ or $\chi_1$, depending on whether we are in the case $p\in\pi^+$ and $q\in\pi^-$ or vice-versa, and assume that $\xi$ does not extend to $G$. Let $\psi$ be an irreducible character of $G$ whose restriction to $S={\rm{PSL}}_2(r^2)$ has $\xi$ as a constituent. Since the diagonal automorphism $\delta$ lies in the inertia subgroup of $\xi$ in $\aut{S}$ (see \cite[Page~2]{White}), we have that the restriction of $\psi$ to $S$ splits as $\xi+\xi^\phi$. Now, the automorphism $\phi$ induces a permutation on the set $\mathcal{C}=\{C_1,\ldots, C_t\}$ of $S$-conjugacy classes of elements of order $p$, where the class $x^S\in\mathcal{C}$ is mapped to $(x^r)^S$, and we consider a set $\mathcal{F}=\{C_1,\ldots,C_s\}$ of representatives for the orbits of this action. Note that, since $\delta$ fixes all the elements of $\mathcal{C}$, a set $\{x_1,\ldots, x_s\}$ of representatives for the $S$-classes in $\mathcal{F}$ is also a set of representatives for the $G$-conjugacy classes of elements of order $p$. 

We claim that the complex numbers $\psi(x_1),\ldots,\psi(x_s)$ are linearly independent over $\Q$. In fact, assume that we have $$\sum_{j=1}^s q_j\psi(x_j)=0,$$ where $q_j$ is a rational number for every $j\in\{1,\ldots, s\}$. Then we get $$0=\sum_{j=1}^sq_j(\xi(x_j)+\xi(x_j^r)),$$ and this is a linear combination with rational coefficients of values of $\xi$ on a set of representatives for the $S$-conjugacy classes in $\mathcal{C}$ (indeed, if $p$ lies in $\pi^-$ then we have $\xi(x_j)=\xi(x_j^r)$ for every $j\in\{1,\ldots s\}$, whereas if $p$ lies in $\pi^+$ then $\xi(x_j)\ne \xi(x_j^r)$ for every $j\in\{1,\ldots s\}$). Since we know these values are linearly independent over $\Q$, it easily follows that $q_j=0$ for all $j\in\{1,\ldots, s\}$. Moreover, for every $j\in\{1,\dots,s\}$, we have that $q$ does not divide $|\Q_p:\Q(\psi(x_j)+\psi(x_j^r))|$ (which is either $2$ or $4$, depending on whether $p$ is in $\pi^-$ or in $\pi^+$, respectively), and again we are done by Lemma~\ref{independence}. This concludes the proof of the claim in~(a).

\smallskip
As for the claim in (b), still referring to Table 2 of \cite{Hertweck2007}, we consider the conjugacy classes of $G$ containing elements of order $r$: there are two of them, represented by elements denoted by $c$ and $d$. Also, we consider the irreducible character $\eta_1$ of $G$. We have $\eta_1(c)=-\frac{1}{2}+\frac{1}{2}i\sqrt{r^f}$ and $\eta_1(d)={\overline{\eta_1(c)}}$, so these values are linearly independent over $\Q$. Moreover, $|\Q(\eta_1(c)):\Q|$ is clearly $2$, hence $|\Q_r:\Q(\eta_1(c))|=(r-1)/2$ is not divisible by $2$ (recall that we are assuming $r^f\equiv -1\,\mod\, 4$), and the same holds for $|\Q_r:\Q(\eta_1(d))|$. Another application of Lemma~\ref{independence} yields that $\ngr{V(\Z(G))}$ does not contain the arc $2\rightarrow r$, and the proof is complete. 
\end{proof}

\section{Metacyclic subgroups of solvable groups: Theorem~\ref{SIPsolvable} and Theorem~\ref{SIPmetacyclic}}

We work next toward an improvement of Corollary~\ref{NPQsolvable} in the spirit of the Subgroup Isomorphism Problem, as described in the introduction. 
We will first prove that if $G$ is a finite solvable group and $V(\Z G)$ has a subgroup isomorphic to a Frobenius group of the form $C_p\rtimes C_{q^k}$, then so does $G$ (Theorem~\ref{SIPsolvable} of the introduction). 

We start with a preliminary result. Observe that the isomorphism type of a group of the kind $C_p\rtimes C_{q^k}$ is uniquely determined by the order of its center, and also by the order of the kernel of the action of $C_{q^k}$ on $C_p$ which defines the semidirect product. 

\begin{lemma}\label{christmaseve} Let $p,q$ be distinct primes, $H$ a finite group and $N$ a normal elementary abelian $q$-subgroup of $H$. Assume that, for a positive integer $k$, the factor group $H/N$ is a Frobenius group of the form $C_p\rtimes C_{q^k}$, and assume that $b\in H$ is an element of order $q^k$ with $\GEN b\cap N=1$. Then there exists a Sylow $p$-subgroup $P$ of $H$ such that $b$ lies in $\norm H P$.
\end{lemma}   

\begin{proof} Let $H$ be a minimal counterexample to the statement, and let $P=\GEN a$ be a Sylow $p$-subgroup of $H$. Observe that, by the Frattini argument, we have $H=N\norm H P$ and, by \cite[Theorem~5.2.3]{Gor}, $N=[N,P]\times \cent N P$; therefore, setting $M=[N,P]$ and $C=\cent N P$, we have $H=MC\norm H P=M\norm H P$. In particular, $M$ and $MP$ are  normal subgroups of $H$, because they are normalized by both $M$ and $\norm H P$. Since $M\cap\norm H P=\norm M P=\cent M P=M\cap C=1$, we see that $\norm H P$ is in fact a complement for $M$ in $H$. 

Now, set $H_0=MP\GEN b$; since $aM$ and $b^{q^{k-1}}M$ are elements of $H_0/M$ that do not commute because even $[aN,b^{q^{k-1}}N]$ is nontrivial, we have that $H_0/M$ is isomorphic to the Frobenius group $C_p\rtimes C_{q^k}$. In fact, $H_0$ is a group which satisfies our assumptions (with respect to $M$ and the same element $b$ of order $q^k$), and if the desired conclusion holds for $H_0$ then it clearly holds also for $H$. In other words, by the minimality of $H$ as a counterexample, we must have $H_0=H$. Note also that $M$ is nontrivial, as otherwise $H$ would not be a counterexample, and so $H$ is a $2$-Frobenius group.  

As the next step, we will count the elements of order $q^k$ in $H$ that normalize some Sylow $p$-subgroup of $H$. Note that, if $y$ is such an element, then $y^{q^{k-1}}$ normalizes a Sylow $p$-subgroup of $H$. Since, by the first paragraph of this proof, the normalizer in $M$ of any Sylow $p$-subgroup of $H$ is trivial, it follows that $\GEN y\cap M=1$. Observe that every Sylow $p$-subgroup of $H$ is contained in $MP$ and, again because $MP$ is a Frobenius group with complement $P$, we get $|\syl p H|=|M|$. In fact, we have $\syl p H=\{P^m\mid m\in M\}$. It follows that the normalizer of every Sylow $p$-subgroup of $H$ is isomorphic to the Frobenius group $C_p\rtimes C_{q^k}$, and therefore the number of subgroups of order $q^k$ of $H$ which normalize a given Sylow $p$-subgroup of $H$ is exactly $p$. We also observe that, by Theorem~15.16 of \cite{Is}, for every element $y\in H\setminus M$ having order $q^k$ we get $|M|=|\cent M y|^{q^k}$; in particular, $\ell=|\cent M y|$ does not depend on the choice of $y$ among the elements of order $q^k$ in $H\setminus M$. Now, we claim that each of these elements of order $q^k$ normalizes precisely $\ell$ distinct Sylow $p$-subgroups of $H$. In fact, assume that $y\in H$ of order $q^k$ normalizes $P_0\in\syl p H$. If, for $m\in M$, we have $y\in\norm H{P_0^m}$, then $[m^{-1},y^{-1}]$ lies in $\norm H{P_0}y^{-1}\cap M=\norm H{P_0}\cap M=1$ and so $m$ lies in $\cent M y$. On the other hand, it is clear that if $y$ centralizes $m$ then $y$ normalizes $P_0^m$, and our claim follows. 

We are ready to conclude the count mentioned in the previous paragraph. Denoting by $\mathcal{C}$ the set of subgroups of order $q^k$ in $H$ that normalize some Sylow $p$-subgroup of $H$, consider the bipartite graph whose vertex set is $\syl p H\cup\mathcal{C}$ (with $\syl p H$ and $\mathcal{C}$ being the two parts of the graph), and where the vertices $P_0\in\syl p H$ and $Q_0\in\mathcal{C}$ are adjacent if and only if $Q_0$ normalizes $P_0$. By the discussion above, the number of edges in this graph can be computed both as $|M|p$ (which is the number of vertices in the part $\syl p H$ multiplied by the number of neighbors of each vertex in $\syl p H$) and $|\mathcal{C}|\ell$ (obtained similarly, from the point of view of the part $\mathcal{C}$). So, we have $|\mathcal{C}|=|M|p/\ell=\ell^{q^k-1}p$. As each subgroup of order $q^k$ contains $q^{k-1}(q-1)$ elements of order $q^k$, we conclude that the number of elements of order $q^k$ in $H$ which normalize some Sylow $p$-subgroup of $H$ is given by $$|\mathcal{C}|q^{k-1}(q-1)=\ell^{q^k-1}pq^{k-1}(q-1).$$

Our aim is then to show that the number above coincides with the total number of elements $y$ of order $q^k$ in $H$ with $\GEN y\cap M=1$. So, let $y$ be such an element: observe that $y$ acts as an automorphism of order $q^k$ on the ${\rm{GF}}(q)$-vector space $M$, whose dimension we denote by~$t$. Hence, up to conjugacy in ${\rm{GL}}_t(q)$, the action of $y$ is represented by an upper unitriangular matrix $Y$ in the canonical Jordan form. It is easily seen that the dimension of $\cent M y$ over ${\rm{GF}}(q)$ coincides with the number of Jordan blocks of $Y$. Denote by $I_d$, $J_d$ the $d$-dimensional identity matrix and the Jordan block with eigenvalue $1$ respectively, and set $E=J_d-I_d$.
Then for a positive integer $h$ we get $$(J_d)^h=\sum_{i=0}^{d-1} {{h}\choose{i}} E^i$$ and, in particular, we see that $(J_d)^{q^k}$ contains the summand $E^{q^k}$ if $d>q^k$. Since in our situation $(J_d)^{q^k}$ must be the identity matrix $I_d$ for every Jordan block of $Y$, we conclude that $d$ is at most $q^k$ for every block of $Y$. But now, recalling that $t=\dim M$ equals $q^k\dim{\cent M y}$ and that $\dim{\cent M y}$ is the number of Jordan blocks of $Y$, we conclude that the dimension of each block is precisely $q^k$.

Finally, for $m\in M$ and $r\in \Z$, consider the element $my^{-r}$ of $H$. Observe that $(my^{-r})^{q^{k-1}}=m\cdot m^{y^r}\cdot m^{y^{2r}}\cdots m^{y^{(q^{k-1}-1)r}}\cdot y^{q^{k-1}r}$. So, if $r$ is a multiple of $q$, we see that $(my^{-r})^{q^{k-1}}$ lies in $M$ and hence $my^{-r}$ is not an element of order $q^k$ with $\GEN {my^{-r}}\cap M=1$. On the other hand, if $r$ is coprime to $q$, then $(my^{-r})^{q^{k-1}}$ cannot lie in $M$, as otherwise the same would hold for $y^{q^{k-1}r}$ and hence for $y^{q^{k-1}}$. As a consequence, if we want to count the elements $z$ of order $q^k$ lying in $M\gen y\in\syl q H$ and such that $\GEN z\cap M=1$, then we have to count the elements of order $q^k$ having the form $my^{-r}$ where $m\in M$ and $r$ is coprime to $q$. We start by focusing on the element $my^{-1}$.

With the methods described in the paragraph above, we see that the order of $my^{-1}$ is $q^k$ if and only if, in an additive notation, we have $$m+Ym+Y^2m+\cdots +Y^{q^k-1}m=0.$$ In other words, $my^{-1}$ has order $q^k$ if and only if $m$ lies in the kernel of the endomorphism of $M$ represented by the matrix $I_t+Y+Y^2+\cdots + Y^{q^k-1}$. Now, it can be computed that  $$I_{q^k}+J_{q^k}+(J_{q^k})^2+\cdots +(J_{q^k})^{q^k-1}=\sum_{i=0}^{q^k-1}\left(\sum_{j=i}^{q^k-1}{j\choose i}\right)E^i,$$ and the $(r,s)$-entry of this matrix depends only on $i=s-r$ (where $1\leq r\leq s\leq q^k$) . Since we have $$\sum_{j=i}^{q^k-1}{j\choose i}={q^k\choose {i+1}}$$ (this equality, not difficult to check, is sometimes referred to as the Christmas Stocking Theorem), it is then clear that $I_t+Y+Y^2+\cdots + Y^{q^k-1}$ is a block-diagonal matrix whose blocks are all equal to the $q^k$-dimensional matrix having all zero entries except the top-right entry, which is $1$. The kernel of the endomorphism associated with this matrix has then a dimension equal to $t-\dim\cent M y$, and so it contains $\ell^{q^k-1}$ elements.

Taking into account that the above calculation yields the same result if we replace $y^{-1}$ by any generator of $\GEN y$, we conclude that the number of elements $z$ of order $q^k$ lying in the Sylow $q$-subgroup $M\GEN y$ of $H$ and with $\GEN z\cap M=1$ is $\ell^{q^k-1}q^{k-1}(q-1)$. Multiplying this result by the number of Sylow $q$-subgroups of $H$, we obtain a total of $\ell^{q^k-1}pq^{k-1}(q-1)$ elements, including our element $b$. The desired conclusion now follows because this number coincides with the number of elements of order $q^k$ in $H$ that normalize some Sylow $p$-subgroup of $H$, and the proof is complete.
\end{proof}

We are ready to prove a slightly stronger form of Theorem~\ref{SIPsolvable}. 

\begin{theorem}\label{SolvableCpCq}
Let $G$ be a finite solvable group, let $p, q$ be distinct primes and let $k$ be a positive integer. Assume that $V(\Z G)$ has two elements $u$, $v$ such that $|u|=p$, $|v|=q^k$, and the group $T$ generated by $u$ and $v$ is isomorphic to the Frobenius group $C_{p}\rtimes C_{q^k}$. Then there exist $a, b \in G$ such that $|a|=p$, $|b|=q^k$, $\langle a,b\rangle\cong T$ and $\varepsilon_b(v)\neq 0$.
\end{theorem}

\begin{proof}
Let $G$ be a counterexample of minimal order. Assume that the Fitting subgroup $\fit G$ of $G$ is not a $p$-group. Then we can take a minimal normal subgroup $N$ of $G$ which is an elementary abelian $r$-group for some prime $r\ne p$. Denoting by $\varphi:V\to V(\Z (G/N))$ the natural homomorphism, we get $T\cap \ker\varphi=1$ and hence $\varphi(T)\cong T$. By our minimality assumption, there exist $x,y \in G$ such that $|\varphi(x)|=p$, $|\varphi(y)|=q^k$, $\GEN{\varphi(x),\varphi(y)}\cong T$, and $\varepsilon_{\varphi(y)}(\varphi(v))\ne 0$.

Since
	\[\varepsilon_{\varphi(y)}(\varphi(v))= \sum_{g\in G\;:\;\varphi(g)\sim \varphi(y)} \varepsilon_g(v),\]
we can find an element $b\in G$ such that $\varepsilon_{b}(v)\ne 0$ and $bN$ is conjugate to $yN$ in $G/N$. Note that, in particular, $|b|$ is a multiple of $q^k$, but in fact we have $|b|=q^k$ by \cite[Theorem~2.3]{Hertweck2007}. Now, let $g\in G$ be such that $bN=(yN)^{gN}$. Clearly $bN$ normalizes $\GEN{x^gN}$, hence $x^g,b\in G$ satisfy the following properties: $|x^g|=p$, $|b|=q^k$, the subgroup $H=\GEN{N,x^g,b}$ is such that $H/N\cong C_p\rtimes C_q^k$ is a Frobenius group, and $\varepsilon_{b}(v)\ne 0$.

If $r\ne q$, then, by the Schur-Zassenhaus Theorem, $N$ is complemented in $H$ by a Hall $\{p,q\}$-subgroup $H_0$, and we can choose $H_0$ so that $b\in H_0\cong H/N\cong T$. Now, if $a$ is an element of order $p$ in $H_0$, then the elements $a$ and $b$ satisfy the conclusions of the statement and $G$ is not a counterexample. On the other hand, if $r=q$, then we are in a position to apply Lemma~\ref{christmaseve} to the subgroup $H$ (note that $\GEN b\cap N=1$ because $|b|=|bN|=q^k$), and again $G$ is not a counterexample. 

In view of the above discussion, we conclude that $\fit G$ must be a $p$-group. By the main theorem in \cite{HertweckOrdersSolvable}, we can choose an element $b\in G$ such that $|b|=q^k$ and $\varepsilon_b(v)\neq 0$; moreover, $b$ acts as an automorphism of order $q^k$ on $\fit G$, because $G$ is solvable and therefore $\cent G{\fit G}=\zent{\fit{G}}$. But then we may apply Lemma~\ref{ActionNT}, and again it turns out that $G$ is not a counterexample. This is the final contradiction that completes the proof. 
\end{proof}

Note that, by \cite{HertweckOrdersSolvable}, if $G$ is solvable and $V(\Z G)$ has a subgroup isomorphic to $C_p\times C_{q^k}$, then so does $G$. This, together with the previous result, solves the Subgroup Isomorphism Problem for subgroups of the form $C_p\rtimes C_{q^k}$ in the two extreme situations when the action of $C_{q^k}$ on $C_p$ is trivial or faithful. It is then natural to consider the following:

\bigskip
\noindent{\bf{Question}}. Let $G$ be a finite solvable group, and assume that $V(\Z G)$ has a subgroup $T$ of the form $C_{p^h}\rtimes C_{q^k}$. Is it true that $G$ has a subgroup isomorphic to $T$?

\bigskip
We provide next a positive answer to this question under the stronger assumption that the group $G$ has a cyclic derived subgroup. Before stating it, we observe what follows.

\begin{remark}\label{ExponentG'}
Let $G$ be a finite group and let $V=V(\Z G)$. We observe that, if an element $v\in V'$ is conjugate in $\Q G$ to an element $g$ of $G$, then $g$ lies in fact in $G'$. 

Indeed, consider the natural algebra homomorphism $\varphi: \Q G\rightarrow \Q(G/G')$ and its restriction to a group homomorphism of $V(\Q G)$ to $V(\Q(G/G'))$. Since $V(\Q(G/G'))$ is an abelian group, $V'$ is contained in the kernel of this group homomorphism; in particular $\varphi(v)=\varphi(1)$, and it easily follows that $v-1$ lies in the kernel of the algebra homomorphism $\varphi$. As this kernel is an ideal, also $g-1$ lies in the kernel of $\varphi$, i.e., $\varphi(g)-\varphi(1)=0$ and therefore $g\in G'$.
\end{remark}

\begin{theorem}\label{isomorphicfaithful}
	Let $G$ be a finite group such that $G'$ is cyclic, and let $p, q$ be distinct primes. Let $u$, $v$ be elements of $V$ such that $|u|=p^h$, $|v|=q^k$ (where h and k are positive integers) and $T=\langle u,v\rangle\cong C_{p^h}\rtimes C_{q^k}$ is nonabelian. Then there exist $a \in G'$ and $b \in G$ such that $\langle a,b\rangle\cong T$, where $b$ is conjugate to $v$ in $\Q G$.
\end{theorem}

\begin{proof}
	By means of contradiction we assume that $G$ is a minimal counterexample to the statement and $T=\GEN{u}\rtimes \GEN{v}$ is as in the hypothesis. Since $T$ is nonabelian, we have that $q$ is a divisor of $p-1$.
	
	We first claim that $G'$ is a $p$-group. Otherwise $G'$ contains a subgroup $N$, normal in $G$, whose order is a prime $r\neq p$. Then, denoting by $\varphi:V\to V(\Z (G/N))$ the natural homomorphism, we get $T\cap \ker\varphi=1$ and hence $\varphi(T)\cong T$. By the minimality of $G$, there are $a\in G'$ and $g\in G$ such that $\GEN{\varphi(a),\varphi(g)}=\GEN{\varphi(a)}\rtimes \GEN{\varphi(g)}\cong T$ and $\varphi(v)$ is conjugate to $\varphi(g)=gN$ in $\Q(G/N)$. On the other hand, by the main theorem of \cite{CaicedoMargolisdelRio2013}, the First Zassenhaus Conjecture holds for $G$; therefore $v$ is conjugate in $\Q G$ to an element $b$ of $G$, so that $bN$ and $gN$ are both conjugate to $\varphi(v)$ in $\Q(G/N)$. Since $gN$ and $bN$ are elements of $G/N$ that are conjugate in $\Q(G/N)$, every irreducible character of $G/N$ takes the same value on $bN$ and $gN$. Hence they are in fact conjugate in $G/N$, and we may assume $g=b$. 
	Let $H=\GEN{N,a,b}$.  Then $H/N\cong T$. 
	If $r\ne q$, then $G$ contains a Hall subgroup of the form $\GEN{a}\rtimes \GEN{b}\cong T$, a contradiction. Otherwise, since $N$ has order $q$ and $p$ is larger than $q$, we have $\GEN{N,a}=N\times \GEN{a}$. Recalling that $b$ is an element of order $q^k$, we have $\GEN{N, b}=N\rtimes\GEN b$ and so $H=(N\times\GEN a)\rtimes\GEN b$; now, as $\GEN{aN,bN}$ is isomorphic to $T$, then so is $\GEN{a}\rtimes \GEN{b}$, against the assumption. This finishes the proof of our claim that $G'$ is a $p$-group.
	 
Next, we consider two different situations, depending on whether $\zent T$ is trivial or not. Assume first $\zent T=1$, which is equivalent to the condition $[u,v^{q^{k-1}}]\ne 1$. As above, we take $b$ in $G$ which is conjugate to $v$ in $\Q G$. If $u^v=u^r$, then $r\not\equiv 1\, \mod\, p$,  therefore $[u,v]=u^{r-1}$ and it has order $p^h$. By Remark~\ref{ExponentG'}, the order of $G'$ is a multiple of $p^h$ and therefore, by hypothesis, $b^{q^{k-1}}$ commutes with $G'$, for otherwise there is $a\in G'$ such that $\GEN{a}\rtimes \GEN{b}\cong T$ and $G$ would not be a counterexample. If $g\in G$, then $g^{b^{q^{k-1}}}=zg$ for some $z\in G'$. Then $g=g^{b^{q^k}}=z^{q}g$ and therefore, $z=1$. This shows that $b^{q^{k-1}}$ is central in $G$ and hence so is $v^{q^{k-1}}$ in $V$, which contradicts the assumption $[u,v^{q^{k-1}}]\ne 1$.

Finally, assume that $\zent T\neq 1$. Since $T$ is nonabelian, we have that $\zent T$ has order $q^t$ for a suitable $0<t<k$; setting $z=v^{q^{k-t}}$, we have $\zent T=\GEN z$. By the fact that the First Zassenhaus Conjecture holds for $G$, we can consider $x\in G$ and $\gamma\in\Q G$ such that $(uz)^\gamma=x$. Note also that the $q$-part $x_q$ of $x$ is $z^\gamma$, and the $p$-part $x_p$  is $u^\gamma$; since $u$ lies in $T'\subseteq V'$, in view of Remark~\ref{ExponentG'}, we then have $x_p\in G'$. Finally, let $b\in G$ and $\delta\in \Q G$ be such that $v^{\delta}=b$. Now, for any $y\in G$, we get $$\varepsilon_y(x_q)=\varepsilon_y(z^\gamma)=\varepsilon_y(z^\delta)=\varepsilon_y(b^{q^{k-t}})$$ (recall that the partial augmentation with respect to a given element $y\in G$ is a class function of $\Q G$). It immediately follows that $b^{q^{k-t}}$ is conjugate to $x_q$ in $G$, and we take $g\in G$ such that ${(b^{q^{k-t}})}^g=x_q$. Clearly the subgroup $H=\langle x_p, b^g\rangle$ has a center of order at least $q^t$, but we claim that $|\zent H|$ is precisely $q^t$ (i.e., $H\cong T$). Set $Z=\oh q{\zent H}$. We will first prove that $|Z|=q^t$, and then that $\zent H=Z$. Indeed, $Z$ centralizes a nontrivial subgroup of the cyclic $p$-group $G'$ and hence it centralizes $G'$; on the other hand, for $P\in\syl p G$, clearly $Z$ acts trivially on $P/G'$ as well, thus $Z$ centralizes $P$. But as $Z$ lies in an (abelian) Hall $p'$-subgroup of $G$, we conclude that $Z$ lies in $\zent G$. If now $|Z|$ is larger that $q^t$, then $b^{q^{k-t-1}}$ would lie in $\zent G$ and hence in $\zent{\Q G}$. But then $v^{q^{k-t-1}}=(b^{q^{k-t-1}})^{\delta^{-1}}=b^{q^{k-t-1}}$ would be central in $V$, contradicting the fact that $v^{q^{k-t-1}}$ does not centralize $u$. Therefore $|Z|=q^t$. As $t<k$ and $b^g$ has order $q^k$, $H$ is not abelian and hence $\zent H=Z$ and $H\cong T$.
\end{proof}

Observe that, if the elements $u$ and $v$ in the above statement generate an abelian subgroup of order $p^hq^k$, then $G$ has a subgroup isomorphic to $\GEN{u,v}$ by \cite{HertweckOrdersSolvable}. This, together with Theorem~\ref{isomorphicfaithful}, completes the proof of Theorem~\ref{SIPmetacyclic} as stated in the introduction.

\bigskip
{\bf Acknowledgment.} This work has been done during a visit of the second author at Dipartimento di Matematica e Informatica (DIMAI) of Universit\` a degli Studi di Firenze, funded by grant MICIU/PRX23/00370. He wishes to thank DIMAI for the hospitality and the Spanish Ministerio de Ciencia, Innovaci\'on y Universidades for the financial support. The authors also wish to thank Leo Margolis for his careful reading and useful comments on this paper.

\enlargethispage{1.5cm}

\end{document}